\documentclass[11pt,reqno]{amsart}
\usepackage{amscd, amsmath, amssymb, amsthm, enumerate, multirow,epigraph}
\usepackage[linktocpage=true]{hyperref}
\hypersetup{colorlinks,linkcolor=blue,urlcolor=blue,citecolor=blue}
\usepackage{chngcntr}
\usepackage[table]{xcolor}
\usepackage{stmaryrd} 
\usepackage{bbm} 
\usepackage{leftidx} 
\usepackage{mathabx} 
\usepackage{rotating} 
\usepackage{wasysym} 
\usepackage[mmddyy]{datetime}
\usepackage{qtree}
\usepackage{float} 
\usepackage[all]{xy}
\usepackage{arydshln}
\makeatletter
\renewcommand*\env@matrix[1][*\c@MaxMatrixCols c]{%
  \hskip -\arraycolsep
  \let\@ifnextchar\new@ifnextchar
  \array{#1}}
\makeatother
\makeatletter
\newcommand{\labeltarget}[1]{\Hy@raisedlink{\hypertarget{#1}{}}}
\makeatother
\usepackage{todonotes}
\usepackage{pgf, tikz, colortbl, braids}
\usetikzlibrary{cd}
\tikzset{
	ch/.style={circle,draw,on chain,inner sep=2pt},
	chj/.style={ch,join},
	every path/.style={shorten >=4pt,shorten <=4pt}
	}

\numberwithin{equation}{subsection}
\counterwithin{table}{subsection}
\counterwithin{figure}{subsection}
\usepackage[margin = 0.9in]{geometry}
\def\<{\langle}
\def\>{\rangle}
\def\0{{\bar 0}}
\def\1{{\bar 1}}

\newcommand{\cC}{\mathcal{C}}

\newcommand{\cF}{\mathcal F}
\newcommand{\cH}{\mathcal{H}}

\newcommand{\cO}{\mathcal{O}}

\newcommand{\CC}{\mathbb{C}}

\newcommand{\Dirr}{\Delta_{\mathrm{irr}}}

\newcommand{\Ext}{\mrm{Ext}}

\newcommand{\ep}{\epsilon}

\newcommand{\fa}{\ensuremath{\mathfrak{a}}}

\newcommand{\e}{\ensuremath{\mathfrak{e}}}
\newcommand{\f}{\ensuremath{\mathfrak{f}}}
\newcommand{\g}{\ensuremath{\mathfrak{g}}}
\newcommand{\fg}{\ensuremath{\mathfrak{g}}}

\newcommand{\fn}{\ensuremath{\mathfrak{n}}}

\newcommand{\fp}{\ensuremath{\mathfrak{p}}}

\newcommand{\ft}{\ensuremath{\mathfrak{t}}}

\newcommand{\Hom}{\mrm{Hom}}

\newcommand{\mrm}{\mathrm}

\newcommand{\ninj}{\nabla_{\mathrm{inj}}}

\newcommand{\tif}{\textup{if }}

\newcommand{\ZZ}{\mathbb{Z}}

\usepackage{enumerate}
\newenvironment{eq}{\begin{equation}}{\end{equation}}

\theoremstyle{definition}

\newtheorem{Def}{Definition}[subsection]
\newtheorem{definition}[Def]{Definition}
\newtheorem{defn}[Def]{Definition}

\newtheorem{ex}[Def]{Example}

\theoremstyle{plain}
\newtheorem{prop}[Def]{Proposition}
\newtheorem{thm}[Def]{Theorem}

\newtheorem{theorem}[Def]{Theorem}
\newtheorem{lemma}[Def]{Lemma}

\newtheorem{cor}[Def]{Corollary}

\begin{document}
\title{Cohomology and Representation Theory for Lie Superalgebras}

\author{\sc Nikolay Grantcharov}
\address
{Department of Mathematics\\ University of Georgia \\Athens\\ GA~30602, USA}
\email{nikolayg@uga.edu}

\author{\sc Daniel K. Nakano}
\address
{Department of Mathematics\\ University of Georgia \\Athens\\ GA~30602, USA}
\thanks{Research of the second author was supported in part by
NSF grant DMS-2401184}
\email{nakano@uga.edu}

\maketitle

\begin{abstract} In this paper the authors describe two general frameworks for studying algebraic supergroup and Lie superalgebra representations. These encompass the 
well-known rational representations for a reductive algebraic group and the parabolic Category ${\mathcal O}$ for semisimple Lie algebras. Many of the important developments will be surveyed 
and a blueprint for these new approaches will be provided. 

The existence of odd elements simply add an additional layer of difficulty. 
The authors prove that the calculation of the sheaf cohomology in the super setting for the rational category is equivalent to understanding the Lie superalgebra cohomology. This 
generalizes the well-known fact that the classical Bott-Borel-Weil theorem is equivalent to Kostant's theorem. Furthermore, for the parabolic Category ${\mathcal O}^{\mathfrak p}$
where ${\mathfrak p}$ is a principal parabolic, the authors show that the characters of simple modules can be obtained by looking at natural generalizations of Kazhdan-Lusztig polynomials given by 
extensions between simple modules and costandard modules arising from the underlying costandardly stratified category. 
\end{abstract} 
\vskip .5cm 
\epigraph{\textit{Cohomology is Representation Theory.}}{---J.L. Alperin}

\section{Introduction}

\subsection{Reductive groups and Lie algebras}\label{S:reductive} Let $G_{\bar 0}$ be a reductive group over ${\mathbb C}$ and ${\mathfrak g}_{\0}$=\text{Lie\ $G_{\0}$} be its Lie algebra. Historically, the rational category 
$\text{Mod}(G_{\0})$ and the (parabolic) Category ${\mathcal O}^{{\mathfrak p}_{\0}}$ have been the main categories of interest in representation theory for these groups and Lie algebras. 
Geometric, categorical and homological tools have been extensively developed to understand the characters of irreducible representations and decomposition numbers. This has led to further developments for Kac-Moody groups and their Lie algebras (cf. \cite{Ku02}). 

Let $B_{\0}$ be a Borel subgroup corresponding to the negative roots for the underlying root system $\Phi_{\bar 0}$. A fundamental result in the area is the Bott-Borel-Weil (BBW) theorem which gives a precise calculation of the sheaf cohomology of the line bundle ${\mathcal L}(\lambda)$ over the flag variety $G_{\0}/B_{\0}$ denoted by ${\mathcal H}^{\bullet}(G_{\0}/B_{\0},{\mathcal L}(\lambda))$. A precise statement 
can be found in  \cite[II 5.5 Corollary]{Jan03}. It is a well-known fact over the complex numbers that the BBW theorem is equivalent to Kostant's theorem on the Lie algebra cohomology. 

Given ${\mathfrak p}_{\0}$ a parabolic subalgebra of ${\mathfrak g}$, the BGG Category ${\mathcal O}^{\mathfrak p_{\0}}$ is a category that has been well-studied. For a standard reference, see 
\cite{Hum08}. The Category ${\mathcal O}^{{\mathfrak p}_{\0}}$ is a highest weight category. Furthermore, the decomposition numbers and the characters of the irreducible representations are known by validity of the Kazhdan-Lusztig conjectures. 

\subsection{Algebraic supergroups and Lie superalgebras} 

There is a vast amount of literature on supergroups and Lie superalgebras representations over the complex numbers. Many of the results have focused on studying the representation theory of a specific 
algebra. From a general point of view, much of this activity has fallen under looking at either rational representations for algebraic supergroups or various Category ${\mathcal O}$ theories for Lie superalgebras. 

The rational category for algebraic supergroup $\text{Mod}(G)$ (when $G_{\0}$ is reductive) directly generalizes the rational category for reductive groups over ${\mathbb C}$. However, in general, the rational 
category $\text{Mod}(G)$ is not semisimple, and has a rich cohomology theory. Representations in this category can be studied using techniques from modular representation theory. In particular, techniques from the 
theory of algebraic groups over fields of characteristic $p>0$ transport well into this setting. Boe, Kujawa and Nakano \cite{BKN10a} initiated the study of detecting subgroups/subalgebras in the late 2000's that has been further refined by Serganova and Sherman \cite{SS26}. These detecting and splitting subgroups behave like defect subgroups in modular finite group representation theory. 

A new parabolic category ${\mathcal O}^{\mathfrak p}$  theory has been recently introduced by Lai, Nakano and Wilbert \cite{LNW25} that was inspired by other prior versions developed by Mazorchuk \cite{Maz14} and Chen, Cheng and Coulembier \cite{CCC21}. In the new version, the emphasis has been on using the 
idea of principal parabolic subalgebras as introduced by D. Grantcharov and Yakimov \cite{GY13}  in order to relate ${\mathcal O}^{\mathfrak p}$ to the parabolic category ${\mathcal O}^{{\mathfrak p}_{\0}}$ for the reductive Lie algebra ${\mathfrak g}_{\0}$. 
In particular, one can show that the cohomology ring for ${\mathcal O}^{\mathfrak p}$ 
$$\text{Ext}^{\bullet}_{{\mathcal O}^{\mathfrak p}}({\mathbb C},{\mathbb C}) \cong S^{\bullet}({\mathfrak g}_{\1}^{*})\otimes \text{Ext}^{\bullet}_{{\mathcal O}^{{\mathfrak p}_{\0}}}({\mathbb C},{\mathbb C})$$ 
as $S^{\bullet}({\mathfrak g}_{\1}^{*})$-modules. Here  
$ \text{Ext}^{\bullet}_{{\mathcal O}^{{\mathfrak p}_{\0}}}({\mathbb C},{\mathbb C})$ is the cohomology ring for the parabolic category ${\mathcal O}^{{\mathfrak p}_{\0}}$ for the Lie algebra ${\mathfrak g}_{\0}$. 
The new framework has also facilitated the study of the category and its blocks, along with the development of cohomological and geometric methods to study representations for ${\mathcal O}^{\mathfrak p}$.

\subsection{} The paper is organized as follows. In the following section (Section~\ref{S:Rat}), we present the setting for rational $G$-modules where $G$ is an affine supergroup scheme with examples of 
ones that admit quasi-reductive Lie superalgebras. In this section, we also discuss the comparison theorem involving sheaf cohomology for $G$ and its relationship to the cohomology of vector bundles over $G_{\0}$. 
The first main theorem relating the composition factors of the sheaf cohomology for $G$ and the Lie superalgebra cohomology 
is presented in Section~\ref{S:BBW-Kostant}. This general result encompasses the classical Bott-Borel-Weil theorem and Kostant's theorem on Lie algebra cohomology. 

The theory of detecting subalgebra and splitting subalgebras is elaborated on in Section~\ref{S:grouprep}. Support varieties are described which enable one to give a geometric interpretation for the atypicality of a 
module over the supergroup $G$ or the Lie superalgebra ${\mathfrak g}$.   

In Section~\ref{S:catO}  we give an overview of the parabolic Category ${\mathcal O}^{\mathfrak p}$.  For the general superalgebra case, the blocks are costandardly stratified and the cohomology can live in infinitely many degrees.  An appropriate notion of complexity of modules can be defined in this setting and leads to some very intriguing questions. Finally, the authors tie in the character theory on Category ${\mathcal O}^{\mathfrak p}$ in Section~\ref{S:KL} with extensions between simples and costandard modules. These results are suggestive of the existence of a general Kazhdan-Lusztig theory.

\subsection{Acknowledgements} The paper is dedicated to Jonathan L. Alperin. The first author received his Ph.D. in mathematics from the University of Chicago in 2025, and appreciates Jon's legacy in the department. The second author met Jon when he was an NSF postdoctoral fellow at Northwestern University in the 1990s. Early on, he read Jon's writings on group representation theory especially his article ``Cohomology is Representation Theory" and his book ``Local Representation Theory". Section~\ref{S:grouprep} and Section~\ref{SS:complexity}, indicate how Jon's research perspective has influenced the new directions in Lie superalgebra representation theory.


\section{Rational $G$-modules}\label{S:Rat}

\subsection{Rational Representations} Let $G$ be a smooth affine supergroup scheme over ${\mathbb C}$ and $\text{Mod}(G)$ 
be the category of rational modules for $G$. Let $H$ be a closed subgroup scheme of $G$. One has a natural induction functor 
$\text{ind}_{H}^{G}(-)$ which is left exact. Set $R^{j}\text{ind}_{H}^{G}(-)$ to be the $j$th higher right derived functors of the induction functor. 
The reader is referred to  \cite[Sections 2,4,5]{BrKl03}, \cite[Section 2]{Br04} for general facts about supergroup schemes and induction. 

Let ${\mathfrak g}=\text{Lie } G$ with ${\mathfrak g}={\mathfrak g}_{\0}\oplus {\mathfrak g}_{\1}$. Moreover, let 
$U({\mathfrak g})$ be the universal enveloping superalgebra of ${\mathfrak g}$. Supermodules over the Lie superalgebra ${\mathfrak g}$ 
can be regarded as unital modules for $U({\mathfrak g})$. Since $U({\mathfrak g})$ as a super Hopf algebra, if  $M$ and $N$ are ${\mathfrak g}$-modules one can 
define a ${\mathfrak g}$-module structure on the dual $M^{*}$ and the tensor product $M\otimes N$. Similar statements can be made for tensor product and dual module constructions in 
$\text{Mod}(G)$. 

One of the main tools has been the use of relative cohomology for the pair $({\mathfrak g},{\mathfrak g}_{\0})$. The projective objects are relatively projective $U({\mathfrak g}_{\0})$-modules. 
Furthermore, every $U({\mathfrak g})$-module admits a relatively projective $U({\mathfrak g}_{\0})$-resolution. By using these facts, given ${\mathfrak g}$-modules, 
$M, N$, one can define the relative extension groups $\text{Ext}^{n}_{({\mathfrak g},{\mathfrak g}_{\0})}(M,N)$ by taking a relatively projective $U({\mathfrak g}_{\0})$-resolution for $M$. 
For the purpose of computations, it is important to note that the relative Lie superalgebra cohomology which can be computed using an explicit Koszul type complex. The reader is referred to 
\cite[3.1.8 Corollary, 3.1.15 Remark]{Ku02},  \cite[Section 2.3]{BKN10a} for further details. 

Let $G$ be a quasi-reductive algebraic supergroup scheme over ${\mathbb C}$ where $G_{\0}$ is a reductive algebraic algebraic group. 
The category $\text{Mod}(G)$ is equivalent to locally finite integral modules for $\text{Dist}(G)=U({\mathfrak g})$ (cf. \cite[Corollary 5.7]{BrKl03}). Moreover, 
$\text{Mod}(G)$ is equivalent to ${\mathcal C}_{({\mathfrak g},{\mathfrak g}_{\0})}$, the category of ${\mathfrak g}$-supermodules that are completely reducible over ${\mathfrak g}_{\0}$, whose projectives 
are relatively projective $U({\mathfrak g}_{\0})$-modules. In this case, one can give an interpretation of the rational $G$-cohomology as follows. If $M$ and $N$ are rational $G$-modules then for all $n\geq 0$ 
\begin{equation}
\Ext^{n}_{G}(M,N)\cong \text{Ext}^{n}_{({\mathfrak g},{\mathfrak g}_{\0})}(M,N).
\end{equation}

\subsection{Quasi-reductive Supergroups and Lie Superalgebras}  Let $G$ be a quasi-reductive supergroup and with Lie superalgebra 
${\mathfrak g}={\mathfrak g}_{\0}\oplus {\mathfrak g}_{\1}$. Formally, this means that $G_{\0}$ is reductive and 
the action of $G_{\0}$ on ${\mathfrak g}_{\1}$ differentiates to the adjoint action of ${\mathfrak g}_{\0}$ on ${\mathfrak g}_{\1}$.  
Often times, the Lie superalgebra is also referred to as {\em quasi-reductive} or {\em classical}. A Lie superalgebra ${\mathfrak g}$ is  \emph{basic classical}
if it is classical with a nondegenerate invariant supersymmetric even bilinear form. 

Over the complex numbers, for ``simple" Lie superalgebras not of Cartan type, there are related quasi-reductive supergroup schemes. 
The quasi-reductive ``simple" Lie superalgebras  are not always simple, but are close enough to being simple. The simple finite-dimensional 
Lie superalgebra over ${\mathbb C}$ have been classified by Kac (cf. \cite{Kac77}). 

We have listed below the Lie superalgebras have appeared frequently in the literature and are of general interest:  
\begin{itemize} 
\item $\mathfrak{gl}(m|n)$, $\mathfrak{sl}(m|n)$, $\mathfrak{psl}(n|n)$\ \ \  [{\em general and special linear Lie superalgebras}],
\item $\mathfrak{osp}(m,n)$\ \ \ [{\em ortho-symplectic Lie superalgebras}],
\item $D(2,1,\alpha)$, $F(4)$, $G(3)$\ \ \ [{\em exceptional Lie superalgebras}],
\item $\mathfrak{q}(n)$, $\mathfrak{psq}(n)$\ \ \ [{\em queer Lie superalgebras}],
\item ${\mathfrak p}(n)$, $\widetilde{{\mathfrak p}}(n)$\ \ \ [{\em periplectic Lie superalgebras}].
\end{itemize} 
For the queer Lie superalgebras, $\mathfrak{q}(n)$ will be the Lie superalgebra with even and odd parts $\mathfrak{gl}_n$, while $\mathfrak{psq}(n)$ is the corresponding simple subquotient of $\mathfrak{q}(n)$  
The periplectic Lie superalgebras include ${\mathfrak p}(n)$ with even component $\mathfrak{sl}_{n}$ and its enlargement $\widetilde{{\mathfrak p}}(n)$ with  even component $\mathfrak{gl}_n$.

 \subsection{Comparison of Sheaf Cohomology} Let $G$ be a supergroup scheme and $H\leq G$ where $H$ is a closed subgroup scheme.  Galban and Nakano \cite[Theorem 3.1.1]{GN23} showed there exists a spectral sequence converging to 
 $R^{\bullet}\text{ind}_{H}^{G} M$ when considered as a $G_{\0}$-module in terms of sheaf cohomology groups $R^{\bullet}\text{ind}_{H_{\0}}^{G_{\0}}(-)$. 
 
\begin{theorem}\label{T:spectralseq1} Let $G$ be a supergroup scheme and $H$ be a closed subgroup scheme of $G$, with ${\mathfrak g}=\operatorname{Lie }{\mathfrak g}$ and
${\mathfrak h}=\operatorname{Lie }{\mathfrak h}$. If $M$ be a $H$-module then there exists a spectral sequence 
$$E_{1}^{i,j}=R^{i+j}\operatorname{ind}_{H_{\0}}^{G_{\0}} [M\otimes \Lambda^{j}({\mathfrak g}_{\1}/{\mathfrak h}_{\1})^{*}] \Rightarrow [R^{i+j}\operatorname{ind}_{H}^{G}M]|_{G_{\0}}.$$ 
\end{theorem} 

We note that the odd elements are built into the spectral sequence via tensoring $M$ by $\Lambda^{j}({\mathfrak g}_{\1}/{\mathfrak h}_{\1})^{*}$. The spectral sequence has been used to 
reprove many of the important vanishing results in \cite{GGNW21}. Furthermore, versions of Kempf's vanishing theorem and the Bott-Borel-Weil theorem have been established for generic (large) weights. 
See \cite[Theorem 5.2.2, Theorem 5.3.1]{GN23}. 

\section{Bott-Borel-Weil and Lie superalgebra cohomology}\label{S:BBW-Kostant}

\subsection{} In this section, let $G$ be a quasi-reductive algebraic supergroup scheme. Moreover, let $B\cong T\ltimes U$ be a closed subgroup scheme of $G$ where $T\subseteq G_{\0}$ is a split 
toral subgroup scheme and $U$ is a unipotent group scheme. 

Let $\lambda\in X(T)$ be a one-dimensional representation for $T$ viewed as a $B$-module by inflation. For $n\geq 0$, consider $R^{n}\text{ind}_{B}^{G}\lambda$ which is a rational $G$-module and 
the $n$-th higher right derived functor of $\text{ind}_{B}^{G} \lambda$. 

Over the complex numbers the category $\text{Mod}(G)$ has enough projective modules. Let $\Lambda_{+}$ index the simple modules in $\text{Mod}(G)$. For each $\sigma \in \Lambda_{+}$, let 
$L(\sigma)$ be the simple module, and $P(\sigma)$ be its projective cover. The following theorem relates the composition factor multiplicities of 
$R^{n}\text{ind}_{B}^{G}\lambda$ with the Lie superalgebra cohomology for ${\mathfrak u}=\text{Lie }U$. 

\begin{theorem}\label{T:sheafucoho} Let $G$ be a quasi-reductive algebraic supergroup scheme with a closed subgroup scheme $B=T\times U$ where $T\subseteq G_{\0}$ is a split 
toral subgroup scheme and $U$ is a unipotent group scheme. Set ${\mathfrak u}=\operatorname{Lie }U$. Then for $\lambda\in X(T)$ and $\sigma\in \Lambda_{+}$, 
$$[R^{n}\operatorname{ind}_{B}^{G} \lambda: L(\sigma)]=\dim \operatorname{H}^{n}({\mathfrak u},P(\sigma)^{*})_{-\lambda}.$$ 
\end{theorem} 

\begin{proof} First observe that 
$$[R^{n}\operatorname{ind}_{B}^{G} \lambda: L(\sigma)]=\dim \operatorname{Hom}_{G}(P(\sigma),R^{n}\text{ind}_{B}^{G} \lambda).$$ 
Now there exists a spectral sequence (cf. \cite[Theorem 3.4.1]{GN23}), 
$$
E_{2}^{i,j}=\text{Ext}^{i}_{G}(P(\sigma),R^{j}\text{ind}_{B}^{G}\lambda)\Rightarrow \text{Ext}^{i+j}_{B}(P(\sigma),\lambda). 
$$ 
Since $P(\sigma)$ is projective, this spectral sequence collapses and yields the isomorphism 
$$
\operatorname{Hom}_{G}(P(\sigma),R^{n}\text{ind}_{B}^{G} \lambda)\cong \text{Ext}^{n}_{B}(P(\sigma),\lambda). 
$$ 
for $n\geq 0$. 

Next consider the Lyndon-Hochschild-Serre spectral sequence for $B\cong T\ltimes U$. 
$$
E_{2}^{i,j}=\text{Ext}^{i}_{T}({\mathbb C},\text{Ext}^{j}_{U}(P(\sigma),\lambda))\Rightarrow \text{Ext}^{i+j}_{B}(P(\sigma),\lambda). 
$$
Modules over $T$ are completely reducible so this spectral sequence collapes and one has 
$$
 \text{Ext}^{n}_{B}(P(\sigma),\lambda)\cong  \text{Ext}^{n}_{U}(P(\sigma),\lambda)^{T}\cong  \text{Ext}^{n}_{U}(P(\sigma),{\mathbb C})_{-\lambda}
 \cong \operatorname{H}^{n}(U,P(\sigma)^{*})_{-\lambda}. 
 $$  

Finally, one can apply the result of Masuoka and Oka \cite[Theorem 3.2]{MO05} for unipotent supergroups (in characteristic zero)  to make the identification: 
$$ \operatorname{H}^{n}(U,P(\sigma)^{*})_{-\lambda} \cong \operatorname{H}^{n}({\mathfrak u},P(\sigma)^{*})_{-\lambda}. 
$$
\end{proof}

\subsection{} The result in the previous section (cf. Theorem~\ref{T:sheafucoho}) is a generalization of a well-known fact about the famous Bott-Borel-Weil theorem and 
Kostant's theorem on Lie algebra cohomology over ${\mathbb C}$. 

Suppose that we are in setting for reductive algebraic groups and semisimple Lie algebras. In this case  $G=G_{\0}$ and $B=B_{\0}$ the negative Borel. Moreover, let $T=T_{\0}$ 
a maximal split torus in $G_{\0}$ with $U=U^{-}$ the unipotent radical of $B_{\0}$. Let $U^{+}$ be the unipotent radical for the opposite Borel, with ${\mathfrak u}^{\pm}=\text{Lie }U^{\pm}$. 

\begin{cor} Let $G_{\0}$ be a reductive algebraic group over ${\mathbb C}$. The Bott-Borel-Weil Theorem for $R^{n}\operatorname{ind}_{B_{\0}}^{G_{\0}} \lambda$ is equivalent to 
Kostant's Theorem for the Lie algebra cohomology of ${\mathfrak u}^{+}$. 
\end{cor} 

\begin{proof} For $\lambda\in X(T)$ and $\sigma\in X(T)_{+}$ (dominant weights), one has by Theorem~\ref{T:sheafucoho}, 
\begin{eqnarray*} 
[R^{n}\operatorname{ind}_{B_{\0}}^{G_{\0}} \lambda: L(\sigma)]&=& \dim \operatorname{H}^{n}({\mathfrak u},L(\sigma)^{*})_{-\lambda} \\
&=&\dim \operatorname{H}^{n}({\mathfrak u},L(-w_{0}\sigma))_{-\lambda}\\
&=&\dim \operatorname{H}^{n}({\mathfrak u}^{+},L(\sigma))_{\lambda}. 
\end{eqnarray*}

The second to third equality follows from comparing Kostant's theorem for ${\mathfrak u}$ to the version for ${\mathfrak u}^{+}$. Finally, one can 
verify the statement of the theorem by using the equality above and the statements of the Bott-Borel-Weil theorem \cite[II 5.5 Corollary]{Jan03} and Kostant's theorem for ${\mathfrak u}^{+}$ 
\cite[Theorem 4.2.1]{UGA09}
\end{proof} 

One may also generalize the aforementioned argument to include Kostant's theorem for the unipotent radical ${\mathfrak u}_{J}$ of a parabolic ${\mathfrak p}_{J}$ with more general version of 
Bott-Borel-Weil. 

\section{Analogies with Finite Groups} \label{S:grouprep}

\subsection{}  Boe, Kujawa and Nakano introduced a new idea for studying representations for classical Lie superalgebras via the 
use of geometric invariant theory (cf. \cite{BKN09,BKN10a, BKN10b}). 
In the case when the action of $G_{\bar{0}}$ on ${\mathfrak g}_{\1}$ is polar and stable, they constructed a two families of detecting subalgebras 
$${\mathfrak e} \leq {\mathfrak f} \leq {\mathfrak g}$$ 
with ${\mathfrak e}={\mathfrak e}_{\bar{0}}\oplus {\mathfrak e}_{\bar{1}}$ and ${\mathfrak f}={\mathfrak f}_{\bar{0}}\oplus {\mathfrak f}_{\bar{1}}$. 
In many cases, ${\mathfrak e}$ is a direct sum of $\mathfrak{q}(1)$s and ${\mathfrak f}$ is a direct sum of $\mathfrak{sl}(1|1)$s.  At this point, it should be mentioned that 
the cohomological results and  the support variety results could be stated via the rational categories for quasi-reductive algebraic supergroups $G$ (resp. $F$ and $E$) where ${\mathfrak g}=\text{Lie }G$ 
(resp. ${\mathfrak f}=\text{Lie }F$ and ${\mathfrak e}=\text{Lie }E$.

The detecting subalgebras have the feature that they are Lie theoretic in nature and they possess the property that they detect the cohomology 
like elementary abelian subgroups and Sylow $p$-subgroups for finite groups. The following diagram illustrates this point. 
\[
\begin{tikzcd}
S^{\bullet}({\mathfrak g}^{*}_{\1})^{G_{\0}}\arrow[r, "\text{res}" ] \arrow[d] & S^{\bullet}({\mathfrak f}^{*}_{\1})^{N} \arrow[r, "\text{res}"] \arrow[d] & S^{\bullet}({\mathfrak e}^{*}_{\1})^{W} \arrow[d] \\
\operatorname{H}^{\bullet}({\mathfrak g},{\mathfrak g}_{\bar{0}},{\mathbb C}) \arrow[r, "\text{res}"] & \operatorname{H}^{\bullet}({\mathfrak f},{\mathfrak f}_{\bar{0}},{\mathbb C})^{N}  \arrow[r,"\text{res}"] & 
\operatorname{H}^{\bullet}({\mathfrak e},{\mathfrak e}_{\bar{0}},{\mathbb C})^{W}. 
\end{tikzcd}
\]
Here $N$ is the normalizer of ${\mathfrak f}_{\1}$ in $G_{\0}$ and $W$ is a finite reflection group. The vertical arrows are isomorphisms which means that the cohomology can realized as rings of invariants. 
These horozontal restiction maps which turn out to be isomorphisms can be viewed as a Lie superalgebra version of the  well-known ``Chevalley restriction theorem" proved for semisimple Lie algebras. 
The picture demonstrates that the ``restriction theorem" for Lie superalgebras arises functorially, and how the odd elements play a essential role in the theory.  

\subsection{} One can define support variety theories in ${\mathcal F}_{({\mathfrak g},{\mathfrak g}_{0})}$ and the corresponding categories 
for the detecting subalgebras. The aforementioned isomorphisms can be used to verify that for any $M\in 
{\mathcal F}_{({\mathfrak g},{\mathfrak g}_{0})}$ there are injective maps of support varieties: 
\begin{equation}\label{E:isomor}
{\mathcal V}_{({\mathfrak e},{\mathfrak e}_{\bar 0})}(M)/W \hookrightarrow {\mathcal V}_{({\mathfrak f},{\mathfrak f}_{\bar 0})}(M)/N \hookrightarrow
{\mathcal V}_{({\mathfrak g},{\mathfrak g}_{\bar 0})}(M). 
\end{equation} 
At the time, when this work was done in \cite{BKN10a}, it was conjectured that the inclusions should be isomorphisms. The conjecture was proved by D. Grantcharov, N. Grantcharov, Nakano and Wu 
\cite[Theorem 1.2.1]{GGNW21}. 

\begin{theorem}\label{T:isosupports}  Let $\g$ be a simple classical Lie superalgebra and let $M$ be in ${\mathcal F_{(\g,\g_{\0})}}$.
\begin{itemize}
\item[(a)] If $\g$ is stable then the map on support varieties
$$\operatorname{res}^{*}:{\mathcal V}_{(\f,\f_{\0})}(M)/N
\rightarrow {\mathcal V}_{(\g,\g_{\0})}(M)$$ is an isomorphism.
\item[(b)] If $\g$ is stable and polar then the maps on support varieties
$$\operatorname{res}^{*}:
{\mathcal V}_{(\e,\e_{\0})}(M)/W_{\e}\rightarrow {\mathcal V}_{(\f,\f_{\0})}(M)/N
\rightarrow {\mathcal V}_{(\g,\g_{\0})}(M)$$
are isomorphisms, where $W_{\e}$ is a pseudoreflection group.
\end{itemize}
\end{theorem} 

This is reminiscent of the situation with modular representation theory of finite groups. The Quillen stratification theorem \cite[II Proposition 5.7.4]{Ben91} states that one can describe the cohomological support varieties for over a finite group via the support varieties over its elementary abelian subgroups. The support varieties for the elementary abelian subgroups have an explicit ``rank variety" realization. In Theorem~\ref{T:isosupports}, one can identify the support variety 
${\mathcal V}_{(\g,\g_{\0})}(M)$ which can be described as being isomorphic to ${\mathcal V}_{(\e,\e_{\0})}(M)/W_{\e}$. In the superalgebra situation, one has also has a realization of the support variety ${\mathcal V}_{(\e,\e_{\0})}(M)$ as 
a rank variety: 
$${\mathcal V}_{(\e,\e_{\0})}(M)\cong {\mathcal V}^{\operatorname{rank}}_{(\e,\e_{\0})}(M):=\{x\in \e_{\1}:\ M|_{U(\langle x \rangle)}\ \text{is not projective}\}\cup \{0\}.$$ 

\subsection{} Given a basic classical Lie superalgebra ${\mathfrak g}$, the {\em defect} ($\text{def}({\mathfrak g})$), and {\em atypicality} ($\text{atyp}(\lambda)$), are a combinatorial invariants for the algebra and 
irreducible representations that were first defined by Kac and Wakimoto (see \cite{KW94}). The following theorem \cite{BKN09}, \cite{Kuj12} shows that one can give a geometric interpretation for the defect\footnote{The second author asked V. Kac one time whether his notion of defect had anything to do with modular representation theory of finite groups. He said that he had not thought about any connection.}  of ${\mathfrak g}$ and the atypicality of a weight $\lambda$. 

\begin{thm} Let ${\mathfrak g}$ be a basic classical Lie superalgebra and $L(\lambda)$ be a simple module in ${\mathcal F}_{({\mathfrak g},{\mathfrak g}_{\0})}$. Then 
\begin{itemize} 
\item[(a)] $\operatorname{def}({\mathfrak g})=\dim {\mathfrak e}_{\1}$ 
\item[(b)] $\dim  {\mathcal V}_{({\mathfrak g},{\mathfrak g}_{\0})}(L(\lambda))=\operatorname{atyp}(\lambda)$.
\end{itemize} 
\end{thm} 

We note that the support variety ${\mathcal V}_{(\g,\g_{\0})}(-)$ does not detect projectivity in the category ${\mathcal F}_{({\mathfrak g},{\mathfrak g}_{\0})}$ because there are infinitely many irreducible representations in 
${\mathcal F}_{({\mathfrak g},{\mathfrak g}_{\0})}$ and the cohomology ring is too small. 

\section{Parabolic Category $\cO$ for Lie Superalgebras}\label{S:catO}

In the section, we will present the basic construction and fundamental results in \cite{LNW25} for the parabolic Category $\cO$ for Lie superalgebras. Throughout, let 
$\fg=\fg_{\0} \oplus \fg_{\1}$ be a finite-dimensional Lie superalgebra over $\CC$.

\subsection{The Definition}
Let $\fa\subseteq \fg$ be a Lie subsuperalgebra. Assume for the definition below,  $\fg = \fa \oplus \fa'$ as vector spaces for some vector space $\fa'$ which is stable under $\fa$.
Moreover, let $\cC_{(\fg, \fa)}$ be the full subcategory of $\fg$-Mod such that each object $M \in \cC_{(\fg, \fa)}$ is a direct sum of finite-dimensional $\fa$-modules.
When $\fg = \fg_\0$ is a Lie algebra with $\fa = \ft_\0$ being its Cartan subalgebra, $C_{(\fg_0, \ft_0)}$ is the natural category of  weight modules in $\fg_\0$-Mod.

\defn\label{def:O}
Let $\fg$ admits a  decomposition $\fg = \fn^+ \oplus \fa \oplus \fn^-$ of Lie sub-superalgebras. Set ${\mathfrak p}={\mathfrak a}\oplus {\mathfrak n}^{+}$. 
Let $\cO^\fp$ be the full subcategory of $\fg$-Mod such that each module $M \in \cO^\fp$ satisfies the following conditions:
\begin{enumerate}[(O1)]
\item $M$ is finitely generated over $U(\fg)$.
\item $M$ is a direct sum of finite-dimensional $\fa$-modules and is semisimple as ${\mathfrak a}_{\0}$-module.
\item $M$ is locally finite with respect to $\fn^+$.
\end{enumerate}
\enddefn

For many Lie superalgebras, one can construct  interesting examples of $\cO^\fp$ arise from the notion of a principal parabolic subalgebra $\fp$ introduced by Grantcharov and Yakimov. 
The reader is referred to \cite[Section 4]{LNW25}.

\subsection{Important facts about ${\mathcal O}^{{\mathfrak p}}$} 
In this section, we will state some useful facts about the parabolic Category ${\mathcal O}^{\mathfrak p}$. We state the following basic properties. 

\begin{prop}\label{Cat O-Liealg prop}Let ${\mathcal O}^{{\mathfrak p}_{\0}}$ be the parabolic Category ${\mathcal O}$ associated with a complex semisimple Lie algebra $\mathfrak{g}_{\0}$. 
\begin{itemize} 
\item [(a)] If $M \in \cO^\fp$, then $M \in \cO^{\fp_{\0}}$.
\item[(b)] If ${\mathfrak p}$ is  a principal parabolic then every $M\in {\mathcal O}^{\mathfrak p}$ has finite length.
\item [(c)] If ${\mathfrak p}$ is a prinicipal parabolic then the simple modules are universally $\0$-bounded. 
\end{itemize} 
\end{prop}

Moreover, the categorical cohomology 
can be computed using relative Lie superalgebra cohomology. 

\begin{theorem}\label{T:cohomology}
If $\mathfrak p$ is a principal parabolic subalgebra of $\mathfrak g$, then the
homological algebra for $\mathcal O^{\mathfrak p}$ coincides with the relative
cohomology for $(\mathfrak g,{\mathfrak l}_{0})$. For
$M,N \in \mathcal O^{\mathfrak p}$ and $n \geq 0$,
\[
\operatorname{Ext}^n_{\mathcal O^{\mathfrak p}}(M,N)
\cong
\operatorname{Ext}^n_{(\mathfrak g,{\mathfrak l}_{0})}(M,N).
\]
\end{theorem}

 \subsection{Construction of the Intermediate Modules}
Suppose that $\fp$ is a principal parabolic subalgebra corresponding to the principal parabolic subset $P_\cH$ for some functional $\cH\in V^*$.
Then, $\fp = \fa \oplus \fn^+$, where $\fa = \bigoplus_{\alpha \in \Phi^0_\cH} {\mathfrak g}_{\alpha}$ is the corresponding ``Levi subalgebra". Let $\Lambda_{\mathfrak a}$ index the non-isomorphic simple $\fa$-modules in $\cC_{(\fa, \fa_\0)}$.  

For $\lambda_{0}\in \Lambda_{\mathfrak a}$, let $L_{\mathfrak a}(\lambda_{0})$ be the 
corresponding simple module and $I_{\mathfrak a}(\lambda_{0})$ be its injective hull.  A simple module (up to parity) in $\cO^\fp$ can be realized as a quotient of $U({\mathfrak g})\otimes _{U({\mathfrak p})} L_{\mathfrak a}(\lambda_{0})$ for some $\lambda_{0}\in \Lambda_{\mathfrak a}$. Considering parity, set $\Lambda:= \Lambda_{\mathfrak a} \times \{0,1\}/\sim$, where two simple modules are  identified when the module and its parity change are isomorphic. For $\lambda\in \Lambda$, let $L(\lambda)$ be the corresponding simple module in $\cO^\fp$. 

One can put a proset structure on $\Lambda$ as follows. Let $\mu, \lambda \in \Lambda$ and $\sigma_{1}, \sigma_{2}$ be weights such that 
$I_{\mathfrak a}(\mu_{0})$ (resp. $I_{\mathfrak a}(\lambda_{0})$) has weights $\gamma$ such that $\overline{\gamma}=\overline{\sigma}_{1}$ (resp. $\overline{\gamma}=\overline{\sigma}_{2}$). 
Then $\Lambda$ is a proset via
\eq\label{eq:Ofaposet}
(\mu_{0}, i) \leq (\lambda_{0}, j) \in \Lambda \quad \Leftrightarrow \quad i=j, \ \sigma_{1}  \leq \sigma_{2} \quad \Leftrightarrow  \quad i=j,\ \overline{\sigma}_{1}  \leq \overline{\sigma}_{2}.
\endeq

Let $\cF$ be the functor that takes the largest submodule in ${\mathcal O}^{\mathfrak p}$. For $\lambda =(\lambda_0, i)\in \Lambda$, 
denote the (co)induced modules from relative injectives and irreducibles, respectively, by 
 \eq\label{eq:ninjDirr}
\ninj(\lambda) := \Pi^i \cF(\Hom_{U(\fp^-)} (U(\fg), I_\fa(\lambda_0))),
\quad
\Dirr(\lambda) := \Pi^i U(\fg) \otimes_{U(\fp)} L_\fa(\lambda_0).
 \endeq
 
In this setup, one can then make use of the tools of homological algebra, along with the parity change functor $\Pi$ that interchanges the $\ZZ_2$-grading of a $\fg$-module.

\subsection{Fundamental results about $\ninj$-filtrations} 

We start by stating an important Ext-orthogonality property between the costandard and standard modules.

\begin{prop}\label{prop:exto}
Let $\lambda, \mu \in \Lambda$. Then,
$$
\Ext^n_{\cO^\fp}(\Dirr(\lambda), \ninj(\mu)) = \begin{cases} 
\CC &\tif n=0,\  \lambda = \mu;
\\
0 &\textup{otherwise}.
\end{cases}
$$
\end{prop}

For costandardly stratified categories, it is important to find a criterion in order to construct $\ninj$-filtrations. The following results were stated in 
\cite[Section 7.2]{LNW25}. Note that the proposition is a variant of \cite[II Proposition 2.14]{Jan03} originally used for highest weight categories. 
.
\begin{prop}\label{lem:Jantzen}
Let $\fp$ be a principal parabolic subalgebra of $\fg$, and let $V$ be a module in ${\cO^{\fp}}$ which satisfies that
$\Hom_{\cO^\fp}(L(\mu), V) \neq 0$ for some $\mu \in \Lambda$.
Suppose that
\begin{enumerate}
\item[(a)] $\Hom_{\cO^\fp}(L(\nu), V) = 0$ for all $\nu < \mu$,
\item[(b)] $\Ext^1_{\cO^\fp}(\Dirr(\nu), V) = 0$ for all $\nu \leq \mu$,
\end{enumerate}
then $\ninj(\mu) \hookrightarrow V$.
\end{prop}

The following theorem give a cohomological criterion of the existence of $\ninj$-filtrations, and a generalized Brauer-Humphreys reciprocity law for 
$\ninj$-filtrations multiplicities inside an injective indecomposable module in $\cO^\fp$. 

\begin{theorem} \label{T:Jantzenlem}
Let $V\in \cO^\fp$. If $\dim \Hom_{\cO^\fp}(\ninj(\gamma), V) < \infty$ for all $\gamma \in \Lambda$. Then the following are equivalent:
\begin{enumerate}
\item[(a)] $V$ has a $\ninj$-filtration.
\item[(b)] $\Ext^n_{\cO^\fp}(\Dirr(\gamma), V) = 0$ for all $\gamma \in \Lambda$ and $n > 0$.
\item[(c)] $\Ext^1_{\cO^\fp}(\Dirr(\gamma), V) = 0$ for all $\gamma \in \Lambda$.
\end{enumerate}
Furthermore, if $I(\lambda)$ is the injective hull of $L(\lambda)$ in ${\mathcal O}^{\mathfrak p}$, then one has the following reciprocity law: 
$$[I(\lambda): \ninj(\gamma)]=[\Dirr(\gamma):L(\lambda)]$$ 
for $\lambda, \gamma\in \Lambda$. 
\end{theorem} 
 
 \subsection{Costandardly Stratified Categories} The modules for the blocks for the parabolic Category ${\mathcal O}$ for Lie algebras form highest weight categories. In \cite{LNW25}, the definition of {\em costandardly stratified category} was formulated 
 which is a true generalization of the original definition of highest weight categories as originally defined by Cline, Parshall and Scott (cf. \cite{CPS88}). Not all blocks for the parabolic Category ${\mathcal O}$ in the Lie superalgebra setting are highest weight categories, but 
 one can say the following. 
 
 \begin{theorem}\label{T:stratification} If $\fp$ is a principal parabolic subalgebra of $\fg$, then 
every block of Category $\cO^\fp$ is  {costandardly} stratified.
\end{theorem}

\subsection{Complexity}\label{SS:complexity} In 1977, Alperin first defined the complexity of a module in the context of modular representation theory of finite groups. In this context, the finite generation of cohomology allows one to conclude that the complexity any module is finite and 
bounded by the Krull dimension of the cohomology ring. 

The complexity of modules in ${\mathcal F}_{({\mathfrak g},{\mathfrak g}_{\0})}$ was first introduced by Boe, Kujawa and Nakano \cite{BKN10b} via rates of growth of the minimal projective resolutions. 
For Category ${\mathcal O}^{\mathfrak p}$, the projectives are infinite-dimensional, but one can define the {\em complexity} for $M \in \cO^\fp$ as  
\eq\label{def:cM}
c_\fp(M) := r(\Ext^{\bullet}_{\cO^\fp}(M, {\textstyle \bigoplus_{L \in \textup{Irr}\cO^\fp}} L )  ),
\endeq
Here $r(-)$ is the rate of growth of the dimensions. 

For Lie algebras, the Category ${\mathcal O}$ has finite homological dimension so the complexity of any module is zero. However, the Category ${\mathcal O}^{\mathfrak p}$ for Lie superalgebras has modules which admit infinite projective resolutions and the complexity 
can quite large, but finite as is shown in the following theorem. 
 
 \begin{theorem}\label{T:complexity}
If $\fp$ is a principal parabolic subalgebra of $\fg$ and $M \in \cO^\fp$ then 
$$c_\fp(M)\leq \dim \fg_\1.$$
\end{theorem}

Note that this is a generalization of a theorem of Coulembier and Serganova for basic classical Lie superalgebras and the standard Borel. See \cite[Theorem 5.6.1]{CS17}

\section{Character theory for Lie superalgebras}\label{S:KL}

\subsection{Characters in terms of $\Ext_{\mathcal{O}^{\mathfrak{p}}}$}Let $\mathfrak{p}$ be a principal parabolic. Recall the triangular decompositions
$$\g=\fn^{-}\oplus\fa\oplus\fn^{+},\;\;\fp=\fa\oplus\fn^+$$
Given $\lambda\in\Lambda$, recall the construction of the Kac induced and coinduced modules from equation \ref{eq:ninjDirr}
 \eq
\ninj(\lambda) := \Pi^i \cF(\Hom_{U(\fp^-)} (U(\fg), I_\fa(\lambda_0))),
\quad
\Dirr(\lambda) := \Pi^i U(\fg) \otimes_{U(\fp)} L_\fa(\lambda_0).
 \endeq
\noindent  Following \cite[7.1]{LNW25}, we let $\mathfrak{t}\subset\mathfrak{a}_{0}$ denote the maximal abelian Cartan subalgebra. Define $R_\ep:=\mathbb{Z}[\ep]/(\ep^2-1)$.
\begin{definition}
    Given $N\in\mathcal{O}^{\mathfrak{p}}$, define the character 
    $$\text{ch}_\ep(N):=\sum_{\lambda\in\mathfrak{t}^*}(\dim N_{\lambda,0}+\ep\dim N_{\lambda,1})e^{\lambda}$$
    And write $\text{ch}(N):=\text{ch}_{\ep=1}(N)$.
\end{definition}

    Define the \textit{Euler characteristic}, respectively, dual Euler characteristic as
    \begin{align}
   c_\mu(N):= \chi(\Ext_{\mathcal{O}^{\mathfrak{p}}}^\bullet(\Dirr(\mu),N))&:=\sum_{n\geq 0}(-1)^n\dim\Ext_{\mathcal{O}^{\mathfrak{p}}}^n(\Dirr(\mu),N)\\
       c^\vee_\mu(N):= \chi^\vee(\Ext_{\mathcal{O}^{\mathfrak{p}}}^\bullet(N,\ninj(\mu)))&:=\sum_{n\geq 0}(-1)^n\dim\Ext_{\mathcal{O}^{\mathfrak{p}}}^n(N,\ninj(\mu))
    \end{align}

The dual Euler characteristic $c^\vee_\mu(N)$ is a finite sum for $N\in\mathcal{O}^{\mathfrak{p}}$, as shown by the following lemma.
\begin{lemma}\label{finiteness of Ext with costandard}
    Suppose $N\in\mathcal{O}^{\mathfrak{p}}$. Then each $\Ext_{\mathcal{O}^{\mathfrak{p}}}^n(N,\ninj(\mu))$ is finite-dimensional and vanishes for sufficiently large $n$. Furthermore, if $L_{\mathfrak{a}}(\mu)$ is a projective module in $\mathcal{C}_{\mathfrak{a},\mathfrak{a}_0}$, then the same conclusion holds for  $\Ext_{\mathcal{O}^{\mathfrak{p}}}^n(\Dirr(\mu),N)$.
\end{lemma}
 \begin{proof}
Let us first prove the first assertion. We have 
     \begin{align*}
     \Ext_{\mathcal{O}^{\mathfrak{p}}}^n(N,\ninj(\mu))&= \Ext_{(\mathfrak{g},\mathfrak{a}_0)}^n(N,\ninj(\mu))&& \text{Theorem~\ref{T:cohomology}}\\
     &=\Ext_{(\mathfrak{p}^-,\mathfrak{a}_0)}^n(N,I_{\mathfrak{a}}(\mu)) &&\text{ Frobenius reciprocity}
     \end{align*}
     
      Now consider the relative Hochschild--Serre spectral sequence
    \[
    E_2^{r,s}
    =
    \Ext_{(\mathfrak a,\mathfrak a_{\bar0})}^r
    \left(
        H_s(\mathfrak n^-;N),
        I_{\mathfrak a}(\mu)
    \right)
    \Rightarrow
    \Ext_{(\mathfrak p^-,\mathfrak a_{\bar0})}^{r+s}
    \left(N,I_{\mathfrak a}(\mu)\right).
    \]
    Since \(I_{\mathfrak a}(\mu)\) is injective in
    \(\mathcal C(\mathfrak a,\mathfrak a_{\bar0})\), all terms with
    \(r>0\) vanish. Hence the spectral sequence collapses.       
 \noindent Thus we see
 \begin{align*}
 \Ext_{\mathcal{O}^{\mathfrak p}}^n(N,\ninj(\mu))
 &\cong
 \Hom_{\fa}\left(H_n(\fn^-;N),I_{\fa}(\mu)\right)\\
 &\subset
 \Hom_{\ft_0}\left(H_n(\fn^-;N),I_{\fa}(\mu)\right)\\
 &\subset\Hom_{\ft_0}\left(
 \Lambda_s^n(\fn^-)\otimes N,I_{\fa}(\mu)
 \right).
 \end{align*}
     Thus, every weight in the corresponding Hom-space is of the form
$\nu-\gamma-\psi_-$, where $\nu$ is a weight of $I_{\fa}(\mu)$,
$\gamma$ is a weight of $N$, and $\psi_-$ is a sum of $n$ negative
weights. For the functional $\cH$ defining the principal parabolic $\mathfrak{p}=\mathfrak{a}\oplus\mathfrak{n}^+$,
define the $\cH$-value of a weight to be its pairing with $\cH$.
Since all weights $\nu$ of $I_{\fa}(\mu)$ have the same
$\cH$-value and the $\cH$-values of the weights of $N$ are bounded
above, the equality $\nu=\gamma+\psi_-$ cannot hold for $n\gg0$.
Hence, weight zero occurs in only finitely many degrees, and each
weight-zero space is finite-dimensional.

The second assertion has a similar proof. By
Theorem~\ref{T:cohomology} and Frobenius reciprocity,
\begin{align*}
\Ext_{\mathcal O^{\mathfrak p}}^n(\Dirr(\mu),N)
&\cong
\Ext_{(\mathfrak g,\mathfrak a_{\bar0})}^n(\Dirr(\mu),N)\\
&\cong
\Ext_{(\mathfrak p,\mathfrak a_{\bar0})}^n
(L_{\mathfrak a}(\mu),N).
\end{align*}
The relative Hochschild--Serre spectral sequence
\[
E_2^{r,s}
=
\Ext_{(\mathfrak a,\mathfrak a_{\bar0})}^r
\bigl(L_{\mathfrak a}(\mu),H^s(\mathfrak n^+;N)\bigr)
\Longrightarrow
\Ext_{(\mathfrak p,\mathfrak a_{\bar0})}^{r+s}
\bigl(L_{\mathfrak a}(\mu),N\bigr)
\]
collapses by projectivity of $L_{\mathfrak a}(\mu)$.
Consequently,
\[
\Ext_{\mathcal O^{\mathfrak p}}^n(\Dirr(\mu),N)
\cong
\Hom_{\mathfrak a}
\bigl(L_{\mathfrak a}(\mu),H^n(\mathfrak n^+;N)\bigr).
\]
Now $H^n(\mathfrak n^+;N)$ is a subquotient of
$\bigl(\Lambda_s^n(\mathfrak n^+)\bigr)^*\otimes N$,
whose weights are $\gamma-\psi_+$, where $\gamma$ is a
weight of $N$ and $\psi_+$ is a sum of $n$ positive weights.
The same weight argument as above therefore proves
finite-dimensionality and vanishing for $n\gg0$.
 \end{proof}

\begin{prop}\label{Dirr basis}
Suppose $N\in\mathcal{O}^{\mathfrak{p}}$. Then
\begin{equation}\label{character with c^vee}
\text{ch}_{\ep}(N)=\sum_{\mu\in\Lambda}c^\vee_\mu(N)\text{ch}_{\ep}(\Dirr(\mu))
\end{equation}
If $\mathfrak{a}=\mathfrak{a}_0$, then the dual statement holds:
\begin{equation}\label{character with c}
\text{ch}_{\ep}(N)=\sum_{\mu\in\Lambda}c_\mu(N)\text{ch}_{\ep}(\ninj(\mu))
\end{equation}
\end{prop}
\begin{proof}
Recall we established Ext-orthogonality between the standard and costandard objects:
\begin{equation*}
\Ext_{\mathcal{O}^{\mathfrak{p}}}^n(\Dirr(\lambda),\ninj(\mu))\simeq\begin{cases}	\mathbb{C}&\text{ if $n=0$ and $\lambda=\mu$}\\ 0&\text{ else}	\end{cases}
\end{equation*}
By Lemma \ref{finiteness of Ext with costandard}, $c^\vee_\mu(M)$ is a finite sum for each $M$. Hence it is a linear functional on the Grothendieck group. Furthermore, $[\Dirr(\mu)]$ form a topological basis of the Grothendieck group and Ext-orthogonality implies
$$c^\vee_\mu(\Dirr(\lambda))=\delta_{\lambda,\mu}.$$
Thus the $c^\vee_\mu$ are coordinate functionals dual to the basis $[\Dirr(\mu)]$ and we conclude
\begin{equation}
[N]=\sum_{\mu\in\Lambda}c^\vee_\mu(N)[\Dirr(\mu)]\Rightarrow\text{ch}_\ep(N)=\sum_{\mu\in\Lambda}c^\vee_\mu(N)\text{ch}_\ep(\Dirr(\mu))
\end{equation}
In the case $\mathfrak{a}=\mathfrak{a}_0$, the category $\mathcal{C}_{\mathfrak{a},\mathfrak{a}_0}$ is semisimple, hence $L_{\mathfrak{a}}(\mu)$ is projective and we may use Lemma \ref{finiteness of Ext with costandard} again to deduce $c_\mu$ induces a linear functional on the Grothendieck group. By Ext-orthogonality, $c_\mu(\ninj(\lambda))=\delta_{\lambda,\mu}$, so the second assertion follows as well.
\end{proof}
The condition $\mathfrak a=\mathfrak a_{\bar0}$ holds for many
familiar parabolics. For example, in $\mathfrak{gl}(m|n)$, take
\(
\mathfrak p=
\left\{
\left(\begin{smallmatrix}A&B\\0&D\end{smallmatrix}\right)
: A\in\mathfrak p_m,\ D\in\mathfrak p_n,\ B\text{ arbitrary}
\right\},
\)
where $\mathfrak p_m\subseteq\mathfrak{gl}(m)$ and
$\mathfrak p_n\subseteq\mathfrak{gl}(n)$ are ordinary parabolics.
Then $\mathfrak a=\mathfrak l_m\oplus\mathfrak l_n$ is purely even,
where $\mathfrak l_m$ and $\mathfrak l_n$ are their respective
Levi subalgebras.

Similarly, for $\mathfrak g=\mathfrak{osp}(m|2n)$, use the
defining form $\operatorname{diag}(I_m,J)$, where
$J=\left(\begin{smallmatrix}0&I_n\\-I_n&0\end{smallmatrix}\right)$.
The principal parabolic defined by
$H=\operatorname{diag}(0_m,I_n,-I_n)$ is
\[
\mathfrak p=
\left\{
\left(\begin{smallmatrix}
A&0&B\\
-B^t&D&S\\
0&0&-D^t
\end{smallmatrix}\right)
:
A\in\mathfrak{so}(m),\quad D\in\mathfrak{gl}(n),\quad
S^t=S,\quad B\in\operatorname{Mat}_{m\times n}(\mathbb C)
\right\}.
\]
Its Levi subalgebra is obtained by setting $B=S=0$, giving
$\mathfrak a\cong\mathfrak{so}(m)\oplus\mathfrak{gl}(n)$.
In particular, $\mathfrak a=\mathfrak a_{\bar0}$.

\begin{ex}
Here we show that the assumption $\mathfrak a=\mathfrak a_0$ in Equation~\eqref{character with c} cannot be omitted in general. Consider the principal block of $\mathfrak{g}=\mathfrak{q}(1)$. This block is equivalent to finite-dimensional representations over $A:=\mathbb{C}[x]/x^2$. The simples are $L=\mathbb{C},\Pi L=\Pi\mathbb{C}$. We find $\Dirr(L)=L$ and $\ninj(L)=\Pi A$ and consequently a minimal projective resolution of $L$ is 

\[
\cdots
\xrightarrow{x}
\Pi A
\xrightarrow{x}
A
\xrightarrow{x}
\Pi A
\xrightarrow{x}
A
\longrightarrow L
\longrightarrow0,
\]
 
\noindent We deduce that
\[
\Ext_A^n(L,L)\cong
\begin{cases}
\mathbb C, & n \text{ even},\\
0, & n \text{ odd},
\end{cases}
\qquad\text{and}\qquad
\Ext_A^n(\Pi L,L)\cong
\begin{cases}
0, & n \text{ even},\\
\mathbb C, & n \text{ odd}.
\end{cases}
\]
Consequently, for $N=L$, the coefficients of
$\operatorname{ch}_{\ep}(\ninj(L))$ and
$\operatorname{ch}_{\ep}(\ninj(\Pi L))$ would be
$1+1+1+\cdots$ and $-1-1-1-\cdots$, respectively.  Hence the right-hand side of
Equation~\eqref{character with c} is not well defined,
whereas the left-hand side is just $1$.

\noindent By contrast, $A$ and $\Pi A$ are injective, while
\[
\Hom_A(L,\Pi A)\cong\mathbb C,
\qquad
\Hom_A(L,A)=0.
\]
Thus $c^\vee_L(L)=1$ and $c^\vee_{\Pi L}(L)=0$, and
Equation~\eqref{character with c^vee} gives
\[
\operatorname{ch}_{\ep}(L)=1\cdot1+0\cdot\ep=1.
\]
\end{ex}

We anticipate that Equation~\ref{character with c} admits a
$q$-analogue for general $\mathcal O^{\mathfrak p}$, perhaps
with coefficients
$c_{q,\mu}(N):=\sum_{j\geq0}(-q)^j
\dim\operatorname{Ext}^{j}_{\mathcal O^{\mathfrak p}}
\bigl(\Delta_{\mathrm{irr}}(\mu),N\bigr)$,
such that the resulting character expression has a well-defined
limit as $q\to1$, recovering $\operatorname{ch}_{\epsilon}(N)$.

We conclude by noting that Proposition~\ref{Dirr basis} places the
coefficients $c_\mu^\vee(M)$ and $c_\mu(M)$ in direct analogy with
the Euler characteristics underlying Kazhdan--Lusztig theory for
ordinary category $\mathcal O$; see
\cite[Proposition~6.14 and Theorem~8.11(b)]{Hum08}.
Thus it is natural to expect the existence of a \textit{super} Hecke algebra governing these Kazhdan-Lusztig polynomials. In particular,
it would be interesting to extend Brundan's canonical-basis
approach \cite{Br03,Br04} to the general categories
$\mathcal O^{\mathfrak p}$ considered here.

\end{document}